\documentclass[a4paper,11pt]{amsart}
\usepackage{amssymb}
\usepackage{graphicx}
\usepackage{psfrag}
\usepackage{amscd}
\usepackage[latin1]{inputenc}
\newcommand{\comm}[1]{}

\def\({\left(}
\def\){\right)}
\def\oli{\overline}

\def\raw{\rightarrow}

\def\no={\neq}
\def\sm{\setminus}

\def\C{{\mathbb C}}
\def\D{{\mathbb D}}

\def\R{{\mathbb R}}

\def\Z{{\mathbb Z}}

\def\BB{{\mathcal B}}
\def\CC{{\mathcal C}}

\def\GG{{\mathcal G}}

\def\MM{{\mathcal M}}

\def\OO{{\mathcal O}}

\def\RR{{\mathcal R}}

\def\UU{{\mathcal U}}

\def\al{\alpha}
\def\be{\beta}
\def\ga{\gamma}

\def\vep{\varepsilon}

\def\th{\theta}

\def\la{\lambda}

\def\Th{\Theta}
\def\La{\Lambda}

\theoremstyle{plain}

\newtheorem{Thm}{Theorem}[section]

\newtheorem{Lem}[Thm]{Lemma}

\newtheorem{Cor}[Thm]{Corollary}

\theoremstyle{remark}

\newtheorem{Def}[Thm]{Definition}

\newtheorem*{Prb}{Problem}

\begin{document}

\title[Shared matings in $V_2$]{Shared matings in $V_2$}

\author{Magnus Aspenberg}
\email{magnusa@maths.lth.se }
\address{Lund University, Center for Mathematical Sciences, Box 118, 221 00 Lund, Sweden.}

\maketitle

\begin{abstract}
We give a new constructive method to prove existence of shared matings in the special class $V_2$ consisting of rational maps with a super-attracting $2$-cycle (up to M\"obius conjugacy). The proof does not use Thurston's Theorem on branched coverings on the Riemann sphere. The background to this paper is the master thesis of L. Pedersen \cite{Lars}, where one special shared mating was studied. 
\end{abstract}

\section{Introduction}

The idea of mating was introduced by J. Hubbard and A. Douady to partially parameterize the space of rational maps (of degree $2$ or higher) by pairs of polynomials of the same degree. Let us here first roughly explain the idea. More precise definitions follow. Take two polynomials $f_1(z) = z^2 + c_1$ and $f_2(z) = z^2 + c_2$, where $c_1,c_2$ belong to the Mandelbrot set $\MM$. Suppose that they have locally connected Julia sets. Then try to glue their filled Julia sets together along their boundaries in reverse order. Under good circumstances one obtains a topological object which is homeomorphic to a sphere. If then the dynamics on this sphere can be realized by a rational map $R$ of degree $2$, then  $R$ is called the mating between $f_1$ and $f_2$. We will make a more precise statement about this later. A. Douady and J. Hubbard conjectured the following. 

\begin{Prb}
As long as $c_1$ and $c_2$ do not belong to conjugate limbs of the Mandelbrot set, then $f_1$ and $f_2$ are mateable.
\end{Prb}

For post-critically finite polynomials, the whole problem with mating in degree $2$ has been settled by T. Lei, M. Rees and M. Shishikura, see \cite{TL} and \cite{MS}.  A main ingredient of this result is Thurston's famous Theorem on branched coverings of the Riemann sphere. This paper is an attempt to introduce new techniques of performing a mating of two post-critically finite quadratic polynomials, without the use of Thurston's theorem. We will focus on a specific class of matings, so called shared matings, studied for instance in the master thesis project \cite{Lars} by L. Pedersen. These shared matings are functions in the special class $V_2$, also studied by V. Timorin in \cite{Timorin}, and \cite{AY} et al. We will use some results from \cite{AY}, without however using any puzzle techniques. In \cite{Timorin} V. Timorin proved existence of matings in terms of laminations on the external boundary of the bifurcation locus of $V_2$, a results that partially covers our main result (Theorem \ref{main}). In fact, \cite{AY} also partially covers our result, since post-critically finite maps are not renormalizable. Shared matings in $V_2$ are also described in \cite{Dudko}. Hence the main result is already in the literature, but we use another approach.  

\subsection{Definition of mating}

Consider two quadratic polynomials $f_1(z) = z^2 + c_1$ and $f_2(z) = z^2 + c_2$, with locally connected Julia sets $J_j$ and where $c_1,c_2$ belong to the Mandelbrot set $\MM$, but do not lay in conjugate limbs of $\MM$. Let $K_j$ be their filled Julia sets $j=1,2$. Then there is a conformal map
\[
\Phi_j : \hat{\C} \sm K_j \mapsto \hat{\C}\sm \oli{\D},
\]
where $\D$ is the (open) unit disk. Since $K_j$ is locally connected $\Phi_j^{-1}$ extends to $\partial \D$ continuously. We define {\em external rays}
\[
R_j(t) = \Phi_j^{-1}(\{re^{it} : r > 1 \}), \quad \text{ for } \quad j=1,2.
\]

We now follow the definition according to Milnor, used for example in \cite{Asp-Roesch}.
Let $S^2$ be the unit sphere in $\C\times\R$.
 Identify each complex plane $\C$ containing $K_i$ (dynamical plane of $f_i$), with the northern hemisphere $\mathbb{H}_+$ for $i=1$ and southern hemisphere $\mathbb{H}_-$ for $i=2$, via the gnomic projections,
\[
\nu_1: \C \rightarrow \mathbb{H}_+ \qquad \nu_2:\C \rightarrow \mathbb{H}_- ,
\]
where $\nu_1(z) = (z,1)/\sqrt{|z|^2+1}$ and $\nu_2(z) = (\bar{z},-1)/\sqrt{|z|^2+1}$. This makes $\nu_2$ equal to $\nu_1$ composed with a 180 degree rotation around the $x$-axis.

It is now not hard to check that the ray $\nu_1(R_1(t))$ of angle $t$ in the northern hemisphere land at the point $(e^{2\pi i t},0)$ on the equator (the unit circle in the plane between the hemispheres). Similarly the ray $\nu_2(R_2(-t))$ on the southern hemisphere of angle $-t$ lands at the point $(e^{2\pi i t},0)$ also. The functions $\nu_i \circ f_i \circ \nu_i^{-1}$ from one hemisphere onto itself are well defined. Moreover, if we approach the equator along the two rays with angle $t$ and $-t$ respectively, both maps $\nu_1 \circ f_1 \circ \nu_1^{-1}$ and $\nu_2 \circ f_2 \circ \nu_2^{-1}$ are going to converge to the same map $(z,0) \rightarrow (z^2,0)$ on the equator. Hence we can glue the two maps together along the equator to form a well defined smooth map  from $S^2$ onto itself. This map, denoted by $f_1 \uplus f_2$ is called the {\em formal mating} of $f_1$ and $f_2$.

Define the {\em ray equivalence relation} $\sim_r$ on $S^2$ to be the smallest equivalence relation such that the closure
of the image $\nu_1(R_1(t))$, as well as the closure of
$\nu_2(R_2(-t))$ lies in a single equivalence class.

The question whether the space $S^2/\sim_r$ is still a sphere is answered by Moore's Theorem.

\begin{Thm}[R.L. Moore]
Let $\sim$ be any topologically closed equivalence relation on $S^2$,
with more than one equivalence class and with only connected
equivalence classes. Then $S^2 / \sim$ is homeomorphic to $S^2$ (via some homeomorphism $h$) if and
only if each equivalence class is non separating.
Moreover let $\pi: S^2 \rightarrow S^2 / \sim $ denote the natural projection.
In the positive case above we may choose the homeomorphism
$h: S^2 / \sim \ \rightarrow  S^2$ such that the composite map $h \circ \pi$ is a uniform limit of homeomorphisms.
\end{Thm}

Note that if $c_1$ and $c_2$ would belong to conjugate limbs of $\MM$, then the $\al$-fixed points $\al_1$ and $\al_2$ for $f_1$ and $f_2$ respectively would be landing points for external rays of conjugate angles; i.e. if $R_1(s)$ lands at $\al_1$, then $R_2(-s)$ lands at $\al_2$. Since each $\al$-fixed point has at least two external rays landing at it, the above mentioned ray equivalence relation would separate the sphere. This makes matings between two such polynomials impossible.

Now we can define conformal mating or simply, mating as follows.

\begin{Def}
The two polynomials $f_1$ and $f_2$ are said {\em conformally mateable}, or just mateable, if
 there exist a rational map $R$ and  two semi-conjugacies $\psi_j: K_j \rightarrow \hat{\C}$ conformal on the interior of $K_j$, such that  $\psi_1(K_1)\cup \psi_2(K_2)=\hat{\C}$ and  \[ \forall (z,w)\in K_i\times K_j, \quad \psi_i(z)= \psi_j(w) \iff z \sim_r w.\]

The rational map  $R$ is called  a {\em conformal mating}.
\end{Def}

\subsection{The family $V_2$.} The family $V_2$ is a space of rational maps which have a super-attracting periodic point of order $2$. Up to M\"obius conjugacy, it is the family of rational maps
\[
R_a(z) = \frac{a}{z^2+2z}, \qquad a \in \hat{\C} \sm \{0 \}
\]
We see that the critical points are $\{\infty,-1\}$. Every $R_a$ has a super-attracting cycle $0 \mapsto \infty \mapsto 0$.
We will study maps in this family for which the so called {\em free} critical point $-1$ is strictly pre-periodic and belongs to the Julia set. We will fix $a$ such that $R_a$ is sub-hyperbolic and sometimes suppress the index for the rest of the paper, writing $R=R_a$.

Let $A_0$ be the immediate basin of attraction for $0$, $A_{\infty}$ be the immediate basin of attraction for $\infty$.
Let $\chi_0: A_0 \mapsto \D$ and $\chi_{\infty} : A_\infty \mapsto \D$ be the corresponding B\"ottcher maps for these domains. Of course these sets depend on the parameter. However, unless needed we will not write out this dependence explicit.

In our specific situation,  rational maps from the family $V_2$ are sometimes matings between the basilica polynomial
$f_{-1}(z) = z^2-1$ and some other polynomial $f_c(z) = z^2 + c$, where $c$ does not belong to the $1/2$-limb of the Mandelbrot set. It is obvious that if $R_a$ is a mating then one of the two mated polynomials has to be the basilica polynomial, since $R_a$ has a super-attracting $2$-cycle for each $a \in \C \sm \{0\}$. We will consider matings where $f_c$ is a so called Misiurewicz-Thurston polynomial; namely a polynomial where the critical point $z=0$ is strictly pre-periodic. The idea is to prove that a sub-class of such polynomials are mateable with the basilica polynomial, by constructing the semi-conjugacy on a dense subset of the filled Julia sets, where the ray-equivalence relation is fulfilled automatically. After that we extend these semi-conjugacies and prove that the extensions are continuous. Finally, a technical lemma (Lemma \ref{biac}) is a main part of proving that the semi-conjugacies obey the ray equivalence relation on the full filled Julia sets.

Let $\RR_c(\th)$ and $\RR_+(\th)$ be the external rays of angle $\th$ for the map $f_{c}$ and $f_{-1}$ respectively. For a given Misiurewicz-Thurston polynomial $f_c$ for which $z=0$ is eventually mapped onto the landing point of the ray $\RR_c(1/3)$, suppose that $\RR_c(\th)$ lands at $c$. Consider another Misiurewicz-Thurston polynomial $f_{c'}$, for which $\RR_{c'}(\th')$ lands on $c'$ where $\RR_+(-\th)$ and $\RR_+(-\th')$ lands on the same bi-accessible point in the basilica. We call $f_c$ and $f_{c'}$ a {\em basilica-pair}. Clearly also $z=0$ is eventually mapped onto $\RR_{c'}(1/3)$ under $f_{c'}$. We have the following result.  
\begin{Thm} \label{main}
Let $f_c$ be a Misiurewicz-Thurston polynomial, such that the critical point $z=0$ is eventually mapped onto the landing point of $\RR_c(1/3)$, and suppose that $c \in \MM$ does not belong to the $1/2$-limb. Then $f_c$ and $f_{-1}$ are conformally mateable. Moreover, if $f_{c'}$ and $f_c$ is a basilica-pair, then the mating $R_a$ between $f_c$ and $f_{-1}$ is also a mating between $f_{c'}$ and $f_{-1}$.
\end{Thm}
Such matings as $R_a$ above are called {\em shared matings} meaning that they are matings between two different pairs of polynomials.
We call $K_{-1}$ and the filled Julia set for $f_{-1}$, and let $J(f_{-1})=J_{-1}$ and $J(f_c) = J_c$ be the Julia sets for $f_{-1}$ and $f_c$ respectively.

\subsection*{Acknowledgments}
I would like to thank Wolf Jung for useful comments. 


\section{Construction of $\psi_+$ on a dense subset of $K_{-1}$}

We start in this section with partially defining a semi-conjugacy $\psi_+: K_{-1} \raw \hat{\C}$. Together with another semi-conjugacy $\psi_c: J_c \raw \hat{\C}$ we prove that $R_a$ is indeed a mating between $f_{-1}$ and $f_c$. Since $R_a$ in the end will be a shared mating, it will also be a mating between $f_{-1}$ and $f_{c'}$ and hence there are semi-conjugacies $\psi_- : K_{-1} \raw \hat{\C}$ together with $\psi_{c'}: J_{c'} \raw \hat{\C}$ for this mating. Since the constructions of the semi-conjugacies $\psi_+$ and $\psi_-$ are very similar, we will mainly focus on one of them. In Subsection \ref{psiconstr} we will see more clearly why these two choices arise (the construction of $\psi_c$ and $\psi_{c'}$ are inherited from these).

Recall that the basilica polynomial $f_{-1}(z) = z^2-1$ has a super-attracting cycle of order $2$; $0 \mapsto -1 \mapsto 0$.
Let $\phi_0: B_0 \mapsto \D$ and $\phi_{-1}: B_{-1} \mapsto \D$ be the B\"ottcher maps of the basins $B_0$ and $B_{-1}$ respectively, where $B_0$ contains $z=0$ and $B_{-1}$ contains $z=-1$. First, let $\psi_+$ be defined on the $B_0$ as
\[
\psi_+(z) = \chi_{\infty}^{-1} \circ \phi_0(z)
\]
and similarly for $B_{-1}$,
\[
\psi_+(z) = \chi_{0}^{-1} \circ \phi_{-1}(z)
\]
Since there is a fixed point $\al = \oli{B_0} \cap \oli{B_{-1}}$ and also a fixed point $p_{\al} = \oli{A_0} \cap \oli{A_{\infty}}$ we extend $\psi_+$ to have $\psi_+(\al) = p_\al$.

Now we can define $\psi_+$ by taking pre-images until we hit the critical point for $R_a$, where $a$ is a specific parameter which depends on the map $f_c$. Suppose that the critical point $z=0$ for $f_c(z)$ is strictly pre-periodic such that $f_c^m(0) = \RR_c(1/3)$ for the least possible $m > 0$. Then $R_a$ will also be critically strictly pre-periodic such that $R^m(-1) = p_\al$ for the least possible $m > 0$. Without specifying $R_a$ more than this for the moment (the precise definition is made in Subsection \ref{choiceofa}), we can define $\psi_+$ by taking the first $m-1$ pre-images of $B_0 \cup \{ \al \}$ because the map $\psi_+$ will be injective there. 
To see this, let 
\begin{equation} \label{L0}
L_{n} = \bigcup_{k=0}^{n} f_{-1}^{-k} (B_0 \cup \{ \al \})
\end{equation}
and suppose we already have defined a conjugacy $\psi_+$, mapping $L_n$ onto its image $\psi_+(L_n)$, for some $n < m-1$. By pulling back $L_n$ under $f_{-1}$, we get that $L_{n+1} \sm L_n$  consists of finitely many new Fatou components (iterated pre-images of $B_0$) that are attached to the previous set $L_n$ via a unique iterated pre-image of the fixed point $\al$. The same is true for the map $R_a$; since $n < m-1$, the pre-image of the set $\psi_+(L_{n+1} \sm L_n)$ under $R_a$ consists of equally many Fatou components that are attached to $\psi_+(L_n)$ via a unique iterated pre-image of $p_{\al}$. Clearly there is only one choice of $\psi_+$ that (continuously) conjugates $L_{n+1}$ with its new image. We can continue pulling back until we hit the critical value, i.e. there is an injective conjugacy
\begin{equation} \label{conjL0}
\psi_+: L_{m-1}  \longrightarrow  \bigcup_{n=0}^{m-1} R_a^{-n}(A_\infty \cup \{p_{\al} \}).
\end{equation}

Now, the critical value of $R_a(-1) = -a$ must belong to $\psi_+(L_{m-1})$, but $-1 \notin \psi_+(L_{m-1})$. Let $w_0$ be defined by $\psi_+(w_0) = -a$. In the next step
there will be two points (pre-images of $w_0$) say $w_1$ and $w_2$ which will forced to be mapped onto the same point, namely $-1$ under $\psi_+$. Since both $w_1$ and $w_2$ are bi-accessible, there are four disjoint bubbles, $B_1, B_1', B_2, B_2'$, such that $\oli{B_1} \cap \oli{B_2} = w_1$ and $w_2 = \oli{B_1'} \cap \oli{B_2'}$.
Let $f_{-1}(B_1) = f_{-1}(B_1') = C_1$ and $f_{-1}(B_2) = f_{-1}(B_2') = C_2$. Hence $\oli{C_1} \cap \oli{C_2} = w_0$. We may fix these bubbles $B_1, B_2, B_1', B_2', C_1, C_2$ by saying that $C_1$ has strictly lower generation than $C_2$ (which means that $B_1$ and $B_1'$ also have lower generation than $B_2$ and $B_2'$ respectively). Moreover, since $C_1$ has strictly lower generation than $C_2$ it means that $B_1$ and $B_1'$ belong to $L_0$.

\subsection{Choosing the right $R_a$} \label{choiceofa}
Given $f_c$ we have to find a candidate rational function $R_a$ that is the mating between $f_{-1}$ and $f_c$. Suppose that $z=0$ is the landing point of two rays $\RR_c(\th_1)$ and $\RR_c(\th_2)$ for $f_c$. By assumption $f_c$ is strictly pre-periodic and there is a (least) $m  > 0$ such that $f_c^m(0)$ is a landing point for the ray $\RR_c(1/3)$. This means that $2^m \th_1 \equiv 2^m\th_2 \equiv  1/3$ $ (\mod 1)$. Since $z=0$ is critical, $f_c(0)$ is the landing point of only one ray $\RR_c(2\th_1) = \RR_c(2 \th_2)$, since $2 \th_1 \equiv 2 \th_2$.

We want to find a rational map in the family $V_2$ which has the property that
the critical point $-1$ is strictly pre-periodic and $R^m(-1)$ is the (unique) fixed point where $A_0$ and $A_\infty$ touch, and moreover such that the Fatou set for $R_a$ in some sense has the same topology as the basilica with some identifications. To do this we need some fact about the parameter space of $R_a$.

\begin{Def}
A {\em parabubble} $\BB$ is a simply connected component of parameters $a$ for which $R_a^n(a) \in A_\infty$, for some fixed $n \geq 0$.
\end{Def}

According to \cite{AY}, there is a correspondence between the right half of the basilica and the set of parabubbles. More precisely, the set $K_{-1} \sm \al$ is the union of two connected components. Let $\BB_R$ be the component that contains $z=0$. 
From Lemma 5.9 in \cite{AY} we see that for each bubble $B \in \BB_R$, there is a corresponding parabubble $P$, which is simply connected, such that $a \in P$ if and only if $-a \in B$ (such parabubbles are called capture components). 

Now suppose that $w_0$ is the point where the ray $\RR_+(-2\th_1)$  lands in the basilica (recall that $\RR_+(-2\th_2)$ is the same ray since $2 \th_1 \equiv 2 \th_2$). There are precisely two bubbles $C_1$ and $C_2$ which touch at $w_0$. We have before chosen these such that $C_1$ has lower generation than $C_2$. By Lemma 5.9 in \cite{AY} there are two corresponding parabubbles $P$ and $P'$ such that $a \in P$ implies that the critical value $-a \in \psi_+(C_1)$ and likewise, $a \in P'$ implies $-a \in \psi_+(C_2)$. In fact, $P$ is the predecessor to $P'$ (see definition after Proposition 5.10 in \cite{AY}) and hence $P$ and $P'$ touch at exactly one point, $a \in \partial P \cap \partial P'$. Now, Lemma 5.9 (III) shows that $-a \in \partial \psi_+(C_1) \cap \partial \psi_+(C_2)$. Clearly, $R_a$ is a Misiurewicz-Thurston map with $R_a^m(-1)=p_{\al}$. 

Fix this parameter $a$ for the rest of the paper. 
We will mainly follow from a thorough study of the parameter space of $V_2$, together with the description given in \cite{AY} to inductively construct $\psi_+$. As seen before, it is clear that there is an injective conjugacy on $L_{m-1}$ by (\ref{conjL0}).
To continue this conjugacy by pulling back, the critical point $-1$ will then be an image of the {\em two} pre-images of the corresponding point in the basilica. Hence these two points glue together in the image in a certain orientation preserving way. The main point is that this gluing can be made in two ways.

\subsection{Two choices for the construction of $\psi_+$} \label{psiconstr}
Write $R_a=R$, where $a$ is chosen in the previous section. Since $R$ is a degree $2$ map the point $-1$ meets four disjoint bubbles ordered as $\{F_1, F_2, F_1', F_2'\}$ clockwise,  which are mapped in pairs onto two bubbles $G_1$ and $G_2$. Since $-1$ is critical of order $2$, $R$ maps two opposite bubbles $F_j$ and $F_j'$ onto the same bubble, i.e. $R(F_1)=R(F_1')=G_1$ and $R(F_2)=R(F_2')=G_2$, where $G_1 \neq G_2$. Moreover, from the previous section, we have that $R$ is strictly pre-periodic, such that $R^m(-1)=p_{\al}$. Since  $\psi_+$ is a conjugacy on $L_{m-1}$ we know that $\psi_+(C_1)$ and $\psi_+(C_2)$ are determined from the construction on $L_{m-1}$ above. Let us say that $\psi_+(C_j) = G_j$, $j=1,2$. By the fact that $B_1, B_1'$ belong to $L_{m-1}$, we may further assume that $\psi_+(B_1) = F_1$ and $\psi_+(B_1')=F_1'$. However, there is still some freedom in how to define  $\psi_+(B_2)$ and $\psi_+(B_2')$. Either $\psi_+(B_2)=F_2$ or $\psi_-(B_2) =F_2'$ (for the latter choice we change the notation from $\psi_+$ to $\psi_-$). As we shall see these two choices represent the two semi-conjugacies for the  basilica-pair $f_c$ and $f_{c'}$, defined just before Theorem \ref{main}.

\begin{figure}
 \setlength{\unitlength}{0.1\textwidth}
\centering
   \begin{picture}(10,6)
  \put(3,0){\includegraphics[scale=1.0]{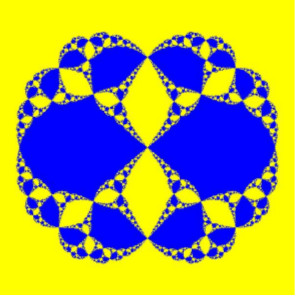}}
  \put(4,2.3){$F_1'$}
  \put(6.4,2.3){$F_1$}
  \put(5.2,3){$F_2'$}
  \put(5.2,1.7){$F_2$}
   \end{picture}
 \caption{Julia set of $R_3$. The middle point is the critical point $z=-1$. It touches the four bubbles $F_1, F_2, F_1', F_2'$. In this picture we also have $A_0 = F_1$ which touches $A_{\infty}$ at $z=1$.}
  \label{R3}
\end{figure}
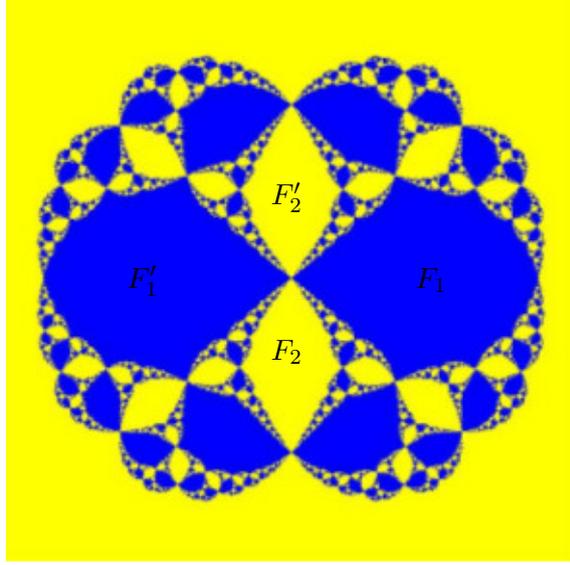



Let us for now fix one such choice, $\psi_+(B_2) = F_2$. No other choices can be made and we get a semi-conjugacy 
\[
\psi_+: \bigcup_{n=0}^{\infty} f_{-1}^{-n} (B_0 \cup \{\al\}) \longrightarrow \bigcup_{n=0}^{\infty} R_a^{-n}(A_{\infty} \cup \{ p_{\al} \})
\]
which is injective on all bubbles in the basilica but two-to-one at certain iterated pre-images of $\al$. Hence the basilica is
mapped under $\psi_+$ into the Riemann sphere with some identifications.

Since $w_1$ and $w_2$ are bi-accessible, there are pairs of external rays landing at them; let us say that $\RR_+(\eta_1)$ and $\RR_+(\eta_1')$ land at $w_1$, and $\RR_+(\eta_2)$ and $\RR_+(\eta_2')$ land at $w_2$. Since $\RR_c(\th_1)$ and $\RR_c(\th_2)$ land on $z=0$ in $J_c$ we may put $\eta_1=-\th_1$ and $\eta_2 = -\th_2$. Suppose that $0 < \eta_2 < \eta_1 < \eta_1' < \eta_2' < 1$, so that the order is fixed. From the choice of $\psi_+$ it also gives a way of gluing together $w_1$ and $w_2$ under $\psi_+$. 
By the above construction we can see that the bubble pairs $B_1 \cup B_2$ and $B_1' \cup B_2'$ are mapped onto the four bubbles $F_1 \cup F_2 \cup F_1' \cup F_2'$ in the written order, i.e. $\psi_+(B_i) = F_i$ and $\psi_+(B_i') = F_i'$, $i=1,2$. 
There are four different non-homotopic paths in $\hat{\C} \sm K_{-1}$ connecting $w_1$ with $w_2$. If we extend $\psi_+$ to one representative $h$ of these paths (take e.g a leaf of
the outside lamination) so that $\psi_+$ maps the full path including the points $w_1$ and $w_2$ onto $-1$ then we get
\[
\psi_+(B_1 \cup B_2 \cup B_1' \cup B_2' \cup \{w_1 \} \cup \{ w_2 \} \cup h)= (F_1 \cup F_2 \cup F_1' \cup F_2' \cup \{ -1 \} )
\]
It is clear that only one of the four paths can make $\psi_+$ satisfy the following additional condition: 
We demand that $\psi_+$ is the uniform limit of a sequence $\psi_n$ of continuous functions on $B_1 \cup B_2 \cup B_1' \cup B_2' \cup \{w_1 \} \cup \{ w_2 \} \cup \{ h \}$, such that the functions $\psi_n$ are injective on $B_1 \cup B_2 \cup B_1' \cup B_2' \cup \{w_1 \} \cup \{ w_2 \}$ (so that the image of $h \cup \{ w_1 \} \cup \{ w_2 \}$ eventually collapses to a point). Hence in a sense the bubble pairs $B_1 \cup B_2$ and $B_1' \cup B_2'$ get glued together in a certain way, via the path $h$. Now fix this path $h$ connecting $w_1$ and $w_2$.

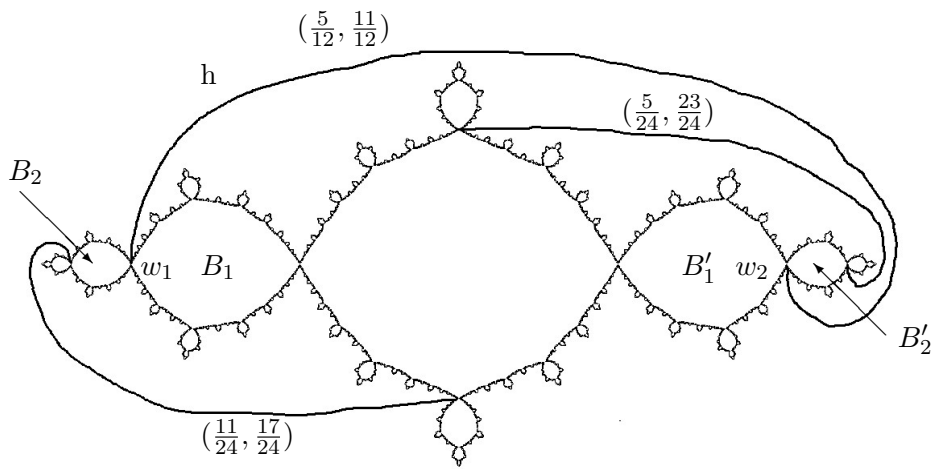
\begin{figure}
   \setlength{\unitlength}{0.1\textwidth}
\begin{center}
   \begin{picture}(10,5)
     \put(1,0){\includegraphics[scale=0.3]{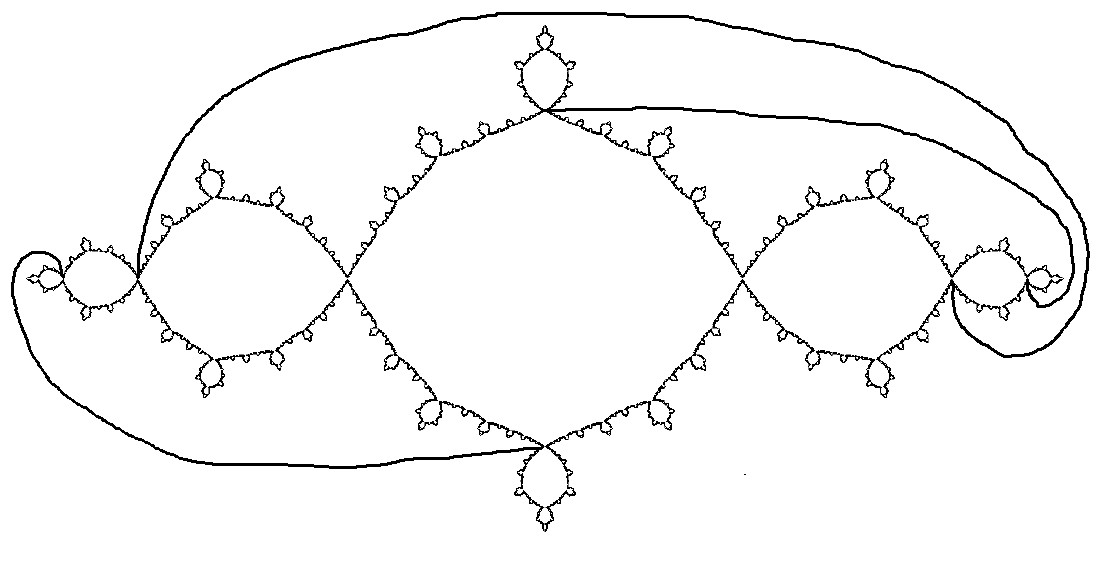}}
     \put(2.5,3.4){h}
    \put(3.3,3.8){$(\frac{5}{12},\frac{11}{12})$}
    \put(2.5,0.38){$(\frac{11}{24},\frac{17}{24})$}
    \put(6,3.1){$(\frac{5}{24},\frac{23}{24})$}
\put(2,1.8){$w_1$}
\put(6.95,1.8){$w_2$}
\put(6.5,1.8){$B_1'$}
\put(2.5,1.8){$B_1$}
     \put(1.0,2.5){\vector(1,-1){0.6}}
\put(0.9,2.6){$B_2$}
     \put(8.2,1.3){\vector(-1,1){0.6}}
\put(8.3,1.2){$B_2'$}
   \end{picture}
\end{center}
 \caption{The path $h$ connects $w_1$ with $w_2$, in this picture such that $h$ lands on the same side as the external rays with angle $\eta_1=5/12$ and $\eta_2=11/12$ respectively. The two other angle pairs are pullbacks of the angle pair $(5/12,11/12)$ corresponding to the outside lamination. The bubbles $B_1,B_2, B_1',B_2'$ are also indicated. They get mapped onto $F_1,F_2, F_1',F_2'$ for the mating, which is $R_3$ in this picture. }
  \label{Basilica-rays}
\end{figure}

A {\em bubble ray} in $\mathring{K}$ is a sequence of (distinct) Fatou components $F_j$ such that  $\overline{F_j} \cap \overline{F_{j+1}} = x_j$ is always a single point. A {\em bubble} is simply a Fatou component (either from $R_a$ or $f_{-1}$). We also define the {\em axis} of a bubble ray $\BB = \{F_k \}$ as the union $\cup \ga_k$ of where each $\ga_k$ is the union of two internal rays in $F_k$ which land on $x_k$ and $x_{k-1}$ (hence $\ga_k$ is a curve inside $F_k$ connecting $x_{k-1}$ with $x_k$). Parameterizing the axis of $\BB$ by an injective function $\ga(t)$, where $t \in [0,\infty)$ with the property that if $\ga(t_1) \in F_k$ and $\ga(t_2) \in F_{k+1}$ then necessarily $t_1 < t_2$, we get a natural way of defining {\em right} and {\em left} of the axis for points in a neighbourhood of the axis if we go along the axis with increasing $t$ for this parameterization. Take a small closed ball $B$ around some $x_j = \overline{F_j} \cap \overline{F_{j+1}}$ such that $B$ only intersects the axis in $\overline{F_j} \cup \overline{F_{j+1}}$. Then the set $B \sm (\overline{F_j} \cup \overline{F_{j+1}})$ consists of two disjoint sets $L$ and $R$. We define $L$ and $R$ as follows. Let $x_1$ be the point $\partial B \cap \ga_{k-1}$ and $x_2$ the point $\partial B \cap \ga_k$ (these points are unique if $B$ is sufficiently small). Take any point $y$ on the boundary of $B$ different from $x_1$ and $x_2$. If the triple $\{x_1,x_2,y \}$ forms a triangle where the corners in this order go counter clockwise we say that $y \in L$ and if the order is clockwise we say that $y \in R$. If an external ray lands at $x_j$ then we say that it lands {\em on the left side of $x_j$} if the ray intersects $L$ and not $R$ and vice versa for any sufficiently small ball $B$.

It is now easy to see that the path $h$ connecting $w_1$ and $w_2$ lands at the same side of $w_j$ as the ray of angle $\eta_j$ does, $j=1,2$. If we consider the other choice for the semi-conjugacy, where $\psi_-(B_2)=F_2'$ instead, corresponding to the shared mating between $f_{c'}$ and $f_{-1}$ (where $f_c$ and $f_{c'}$ is a basilica-pair), we get that $h$ lands on the same side of $w_j$ as the ray of angle $\eta_j'$ does, $j=1,2$. Hence the path $h$ would still connect $w_1$ and $w_2$ but from different {\em sides}. Note that $f_{c'}$ has two rays $\RR_{c'}(\th_1')$ and $\RR_{c'}(\th_2')$ landing on the critical point $z=0$, whereas for $f_c$ these rays are $\RR_c(\th_1)$ and $\RR_c(\th_2)$. We will not further consider the mating with $f_{c'}$ , since the construction for $f_c$ is completely analogous.

\subsection{Laminations for the basilica}
We say that $L_{\OO}$ is the (outside) lamination in $S^1$ defined by connecting the rays $\th_1$ with $\th_2$ (or $\th_1'$ and $\th_2'$ if considering the other shared mating), and taking pre-images under angle doubling. From the previous section, each leaf in $L_{\OO}$ corresponds 
to an identification of two points in the basilica in the following way: First we say that $\th \sim_{L_\OO} \th'$ if $\th$ and $\th'$ belong to the same leaf in $L_{\OO}$. Suppose that the two rays $\RR_+(\th)$ and $\RR_+(\th')$ land at two points $w$ and $w'$ respectively. Then $\psi_+(w) = \psi_+(w')$ if $\th \sim_{L_{\OO}} \th'$. The converse is not true.
Each $w$, which is a landing point of a ray with angle $\th$ belonging to some leaf of $L_\OO$, has to be bi-accessible. Hence there are two angles $\th$ and $\th_1$ such that the rays
$\RR_+(\th)$ and $\RR_+(\th_1)$ land at the same point $w$. Let $\sim_{Biac}$ be the equivalence relation such that $\th \sim_{Biac} \th_1$ if and only if $\RR_+(\th)$ and $\RR_+(\th_1)$ land at
the same point. Clearly this equivalence relation generates in the obvious way the ``ordinary'' basilica lamination.





\section{Construction of $\psi_c$ on a dense subset of $J(f_c)$}

We begin with defining $\psi_2$ which we call $\psi_c$ on a dense subset of the Julia set of $f_c(z) = z^2  + c$. After that we will extend $\psi_c$ to the
full Julia set. Since $f_c$ is a Misiurewicz-Thurston map, $c$ is a landing point of an external ray $\RR_c(\be)$ of a strictly pre-periodic angle 
$\be = p/q$. We have chosen it such that $\be 2^{m-1} = 1/3$ for some $m > 0$. For example
if $f_c(z) = z^2 + i$ we have that $\be = 1/6$.  Let $\th_1$ and $\th_2$ be the angles of the external rays landing on the critical point $z=0$.

\subsection{Construction of the semiconjugacy on a dense subset of the Julia set}
Let
\[
E = \bigcup_{n=0}^{\infty} f_c^{-n}(0).
\]
Since the critical value  is uni-accessible all points in $E$ are
bi-accessible. Recall that the ray $\RR_c(2\th_1) = \RR_c(2\th_2)$ lands on $c$ and that the rays
$\RR_c(\th_1)$ and $\RR_c(\th_2)$ land at $0$. The corresponding rays $\RR_+(-\th_1)$ and $\RR_+(-\th_2)$ for the basilica land on two distinct bi-accessible points $w_1$ and $w_2$. From the previous section we know that $\psi_+(w_1) = \psi_+(w_2)$ and hence it is natural to define
\[
 \psi_c(0) = \psi_+(w_1) = \psi_+(w_2).
\]
Moreover, this directly makes $z=0 \in J(f_i)$ ray equivalent to $w_1$ and $w_2$. We want to extend $\psi_i$ to $E$, by taking pre-images.

Now, $w_1$ and $w_2$ have two pre-images each under $f_{-1}$. Let us denote these pre-images by $w_{11}, w_{12}, w_{21}, w_{22}$. They are identified in pairs under $\psi_+$ according to the
lamination $L_{\OO}$. Suppose that
$\psi_+(w_{12})=\psi_+(w_{11})$ and $\psi_+(w_{21}) = \psi_+(w_{22})$ and let $z_1$ and $z_2$ be the pre-images to $0$ under
$f_c$. The ray-pairs $\RR_c(\th_{j1})$ and $\RR_c(\th_{j2})$ landing on $z_j$ correspond to a ray pair $\RR_+(-\th_{j1})$ and $\RR_+(-\th_{j2})$ landing on $w_{j1}$ and $w_{j2}$ respectively, for $j=1,2$.
There is only one natural way to define $\psi_c(z_1)$ in order to preserve the ray equivalence relation between the pre-images, namely
\begin{align}
 \psi_c(z_1) &= \psi_+(w_{11}) = \psi_+(w_{12}) \\
 \psi_c(z_2) &= \psi_+(w_{21}) = \psi_+(w_{22}).
\end{align}
Continuing in the same manner we obtain a map $\psi_c$ defined on $E$. If we put
\[
E' = E \cup \bigcup_{n=1}^m f_c^n(0)
\]
then we easily to extend $\psi_c$ to $E'$ via the map $\psi_+$. Every point $z$ in the forward orbit of $0$ has landing angles $\th$ for which the landing point $w$ of the ray $\RR_+(-\th)$ is a bi-accessible point. Since $\psi_+(w)$ is well defined from before, we can define $\psi_c$ on these points by letting
\[
\psi_c(z) = \psi_+(w).
\]
Let $p_1$ be the landing point of the ray $\RR_c(1/3)$ and $p_2$ the landing point of the ray $\RR_c(2/3)$. To fix the ideas, let us assume that the critical point $z=0$ is mapped under $f_c$ onto $p_1$ {\em before} $p_2$, meaning that $f_c^{n_j}(0)=p_j$ for the least possible $n_j$, $j= 1,2$ then $n_1 < n_2$. Since the point $p_{\al}=R^m(-1)$ is a fixed point, i.e. $R(p_{\al}) = p_{\al}$, we have to have
\[
 \psi_c(p_1) = \psi_c(p_2)=p_{\al}.
\]

From the construction we see that if $z \in E'$ and $w \in \cup_{n=0}^{\infty} f_{-1}^{-n}(\al)$ (i.e. $w$ is bi-accessible) then $z \sim_r w$ if and only if $\psi_c(z) = \psi_+(w)$.  Moreover, it also follows that if
$w_1, w_2 \in L$ and $w_1 \in \partial \RR_+(\th)$, $w_2 \in \partial \RR_+(\th')$ where $\th \sim_{L_{\OO}} \th'$ then $w_1 \sim_r w_2$.

We no extend $\psi_c$ to
\[
 D =  \bigcup_{n=0}^{\infty} f_c^{-n}(p_1).
\]
Note that $E' \subset D$, so we will concentrate on $D \setminus E'$. Since the critical point $z=0$ is mapped under $f_c$ onto $p_1$ {\em before} $p_2$, one of the two pre-images of $p_1$ is equal to $f_c^{m-1}(0)$, which has to be the landing point of the ray $\RR_c(1/6)$. Now, $p_2$ has two pre-images, namely $p_1$ and another one, say $q$, which is uni-accessible and has to be the landing point of the ray $\RR_c(5/6)$. The ray $\RR_+(-5/6)$ is bi-accessible and lands on some point $r$, so $\psi_+(r)$ is well defined from before. Thus define
\[
\psi_c(q) = \psi_+(r). 
\]
Note that we have already set $\psi_c(f_c^{m-1}(0)) = \psi_+(r)$ from before. Hence we have defined $\psi_c$ on the set $U_1 = f_c^{-1}(p_1) \cup f_c^{-1}(p_2)$.  Now we take pre-images. Let us proceed inductively. Note that all points on $D \sm E'$ are uni-accessible. Suppose we have constructed $\psi_c$ on $U_{k-1}$ where $$U_k = \cup_{j=0}^k f_c^{-j}(p_1) \cup f_c^{-j}(p_2).$$ 
If $k < m$ it means that the set $U_k$ does not contain the critical point. Every point $q_j$ in $U_k \sm U_{k-1}$, $j = 1, \ldots, n_k$, (where $n_k=2^k$ if $k < m$ and $n_k = 2^k-2^{k-m}$ if $k \geq m$) is landing point of a ray $\RR_c(\th_j)$, where $\th_j$ is a pre-image of $1/3$ or $2/3$ under angle doubling. It can easily be seen (by induction for example) that the corresponding set $\Th_k$ of angles $\th_j$ for the points $q_j$ has the property that they can be ordered in pairs $(\th_j,\th_j')$ such that $\RR_+(-\th_j)$ and $\RR_+(-\th_j')$ both land on the same bi-accessible point $r_j$ in the basilica. This means that we have to define 
\begin{equation} \label{ident1}
\psi_c(q_j) = \psi_c(q_j') = \psi_+(r_j),
\end{equation}
where $q_j'$ is the landing point of the ray $\RR_c(\th_j')$. 

We recall that there are two rays $\RR_c(\th_1)$ and $\RR_c(\th_2)$ landing on $z=0$ and where $2\th_1 \equiv 2 \th_2$, i.e. $\RR_c(2\th_1) =\RR_c(2\th_2)$ which lands on $c$. Hence the angle-pairs in $\Th_k$ which give the identification (\ref{ident1}) give rise to more identifications on the set $E$, i.e. the iterated pre-images of the critical point, as described above.

We have constructed $\psi_c$ via $\psi_+$ on the full set $D$. Note also that from the construction we have immediately that two points $z,w \in D$ are ray-equivalent if and only if $\psi_c(z) = \psi_c(w)$. Put
\[
\BB iac = \bigcup_{n=0}^{\infty} f_{-1}^{-n}(\al),
\]
i.e. $\BB iac$ is the set of all bi-accessible points in the basilica. Then for any $z \in D$ and $w \in \BB iac$ we have
that $z \sim_r w$ if and only if $\psi_c(z) = \psi_+(w)$.

Hence we now have two semi-conjugacies which obey the ray-equivalence relation on a dense subsets $\BB iac \subset J(f_{-1})$ and $D \subset J(f_c)$. We now want to extend these functions to the full filled Julia sets.

\section{Extension of $\psi_+$ and $\psi_c$}

From the standard theory of iteration of critically finite rational maps,  there exists an expanding hyperbolic metric outside the post-critical set, if it consists of at least $3$ points (which is the case for $R_a$ in this paper). In fact there is an orbifold metric $\varphi(z)$, and some $\La > 1$ such that for any $z \in \hat{\C} \sm P(R)$ we have
\[
\varphi(R(z)) |R'(z)| \geq \La \varphi(z).
\]
As a consequence of this we get
\begin{Lem} \label{hypmetric}
There exists constants $C > 0$ and $0 < \la < 1$ such that for any connected compact set $E \subset \hat{\C} \sm P(R)$ we have that if $F$ is any component of $R^{-n}(E)$ then
\[
diam (F) \leq C \la^n.
\]
\end{Lem}
This means that the diameter of iterated pre-images of bubbles must tend to zero. 
\begin{Def}
The {\em generation} of a bubble $B$ (for the basilica) is the smallest integer $m \geq 0$ such that $f_{-1}(B)  = B_0$. For the rational map $R_a$ it is the smallest integer $m \geq 0$ such that $R_a^m(B) = A_0$. The generation of a union of bubbles is the smallest generation of a bubble in this union.   
\end{Def} 
We now have a direct consequence of Lemma \ref{hypmetric} we have:
\begin{Cor} \label{landing}
There exist constants $C > 0$ and $0 < \la < 1$ such that if $B$ is any bubble for $R_a$, then
\[
diam(B) \leq c\la^{gen(B)},
\]
where $gen(B)$ is the generation of the bubble. Moreover, every bubble-ray must land at a single point.
\end{Cor}

\begin{Lem} \label{diamep}
Let $B$ be a bubble of the basilica. Then for any $\vep > 0$ there exist only finitely many components $U$ of $K_{-1} \sm B$
for which $\psi_+(U)$ has diameter $diam (U) \geq \vep$.
\end{Lem}
\begin{proof}
Every such component $\psi_+(U)$ is a bubble-ray. Let $gen(U)$ be the lowest generation among the bubbles in this bubble-ray. Then it is clear that for any $K < \infty$ there are only finitely many components $U$ such that $\psi_+(U)$ has $gen(U) \leq K$. But since the diameter of bubbles tend to zero exponentially with the generation, according to the above Corollary, we get that also the sum of diameters of the bubbles in $\psi_+(U)$ tend to zero with the generation. This proves the lemma.
\end{proof}

We also need the following.
\begin{Lem} \label{tips}
For the rational map $R_a$, we cannot have an infinite bubble ray landing at $p_{\al}$. 
\end{Lem}
\begin{proof}
Suppose that there is a infinite bubble ray $\BB$ landing at $p_{\al}$ (the fixed point between $A_{\infty}$  and $A_0$). If $\BB$ is periodic the result follows from Lemma 6.2 in \cite{AY}. If $\BB$ is not periodic it means that all its forward images also have to land at $p_{\al}$. Hence there are infinitely many bubble rays $\BB_k$ landing at $p_{\al}$. This we claim is impossible.

Take a finite bubble ray $\BB_0$ such that its last bubble $B$ of highest generation belongs to  a ball $B(p_{\al},r)$ of radius $r > 0$ around $p_{\al}$. Suppose moreover that $r > 0$ is chosen sufficiently small so that inside this ball there is an (conformal) inverse branch of $R$. Let $\BB_k = R^{-k}(\BB_0)$, where we use the inverse branch of $R$ in the neighbourhood $B(p_{\al},r)$. 
Since there are only finitely many bubbles of any fixed generation there is at least one finite bubble ray $\BB_k'$ that is contained in infinitely many bubble rays $\BB_k$, $k > k'$. Take such $\BB_k'$. Then there is some $m > 0$ such that $R^{-m}(\BB_k')$ contains $\BB_k'$. In other words, $\BB_{k'+m}$ is an extension of $\BB_k'$. Repeating this we see that the bubble rays $\BB_{ml+k'}$ is an increasing sequence of finite bubble rays. The union of all these bubble rays is an infinite bubble ray $\BB_{\infty}$ that has to land at $p_{\al}$, since this point is repelling. Moreover,  $R^m(\BB_{\infty}) = \BB_{\infty}$, so it is periodic. Now apply Lemma 6.2 in \cite{AY} again to reach a contradiction.

\end{proof}

As a corollary, the lemma shows that an infinite bubble ray $\BB$ cannot land on a touching point between two bubbles for $R_a$. These points are precisely the images of the bi-accessible points under $\psi_+$.

These lemmas enables us to extend $\psi_+$ and $\psi_i$ to the filled in Julia sets.

Firstly, since the map $R_a$ is sub hyperbolic, its Julia set is locally connected and hence we may extend $\psi_+$ continuously to the closure of each bubble in $K_{-1}$.

\begin{Def}
A {\em tip} of the Julia set for the basilica is a point $w \in J(f_{-1})$ which does not belong to any bubble. 
\end{Def}

If $z \in K_{-1}$ is a tip of a bubble-ray $\BB$, i.e. $\BB$ consists of bubbles $B_n$ where $dist(z,B_n) \raw 0$ as $n \raw \infty$, then put
\[
\psi_+(z) = \lim\limits_{n \raw \infty} \psi_+(B_n).
 \]
By Corollary \ref{landing} $\psi_+$ is well defined.

We now prove that $\psi_+$ is continuous.
\begin{Lem}
The function $\psi_+$ is continuous.
\end{Lem}

\begin{proof}
We have to prove that if $w_k \raw w$, $w_k \in K_{-1}$, then $\psi_+(w_k) \raw \psi_+(w)$. Suppose that $w$ belongs to the closure of a minimal bubble-ray $\BB = \{B_k\}_{k=0}^N$, starting at $B_0$. Hence $N$ is finite if $w$ belongs to a boundary of a bubble and infinite if $w$ belongs to a tip.  Let $\UU_k$ be the minimal bubble-ray, whose closure contains $w_k$. Let $G_k$ be the (unique) bubble of maximal generation such that $G_k \subset \BB \cap \UU_k$. Set 
\[
\GG_k = ((\UU_k \cup \BB) \sm (\UU_k \cap \BB)) \cup G_k.
\] 
Note that $w, w_k \in \oli{\GG_k}$ for all $k$. We divide the proof in two cases.

{\bf Case 1: $gen(\BB)=\infty$}. By the hyperbolicity of $f_{-1}$ this means that $\GG_k$ has to decrease to a point, as $k \raw \infty$, since $gen(\GG_k) \raw \infty$. Obviously, also $gen(\psi(\GG_k)) \raw \infty$ as $k \raw \infty$. But this means that, by Corollary \ref{landing}, also $diam(\psi_+(\GG_k)) \raw 0$. Since both $w_k$ and $w$ belong to the closure of $\GG_k$, we have immediately that $\psi_+(w_k) \raw \psi_+(w)$. 

{\bf Case 2: $gen(\BB) < \infty$}. Thus $w$ belongs to the closure of a bubble, say $B \subset \BB$. If $\GG_k$ eventually becomes equal to $B$ then it means that $w_k$ eventually belongs to the closure of $B$ and the lemma follows from the construction of $\psi_+$, which is continuous on the closure of bubbles. 

If $\GG_k$ does not eventually become equal to $B$, let $p_k \in \oli{B}$ be the the touching point between $B$ and another bubble in $\GG_k$ attached to $B$. Clearly $p_k \raw w$ as $k \raw \infty$. 

If $p_k = w$ for all sufficiently large $k$, that means that $w$ is an iterated pre-image of $\al$. Hence $w = \oli{B} \cap \oli{B'}$ where $B' \subset \GG_k$ is another bubble of higher generation than $B$. Write 
$\CC_k = \GG_k \sm (B \cup B')$. If $\CC_k = \emptyset$ for sufficiently large $k$ then it means that $w_k \in \partial B'$ and $\psi_+(w_k) \raw \psi_+(w)$ by the continuity on the boundary of bubbles. 
If $\CC_k \neq \emptyset$ for arbitrarily large $k$ then suppose for simplicity that $\CC_k \neq \emptyset$ for all $k$ by passing to a suitable subsequence. Let $q_k$ be the touching point between $\CC_k$ and $B'$. Clearly, $gen(\CC_k) \raw \infty$ and $q_k \raw w$. 
Hence $diam(\psi_+(\CC_k)) \raw 0$ as $k \raw \infty$ by Corollary \ref{landing}. Let $\ga_k$ be the smallest closed arc along the boundary of $B'$ between $q_k$ and $w$. Since $q_k \raw w$, we have that the set $\CC_k \cup \ga_k$ 
is a connected set whose diameter tends to zero, and hence the set is true for $\psi_+(\CC_k \cup \ga_k)$. Since both $w_k$ and $w$ belong to the closure of this set it follows that $\psi_+(w_k) \raw \psi_+(w)$. 

Finally, we consider the case $p_k \neq w$ for all $k$ (by possibly passing to a suitable subsequence). Similar to the above case, let $\CC_k = \GG_k \sm B$ and let $\ga_k$ be the smallest arc along the boundary of $B$ connecting $w$ with $p_k$. Since the generation of $\CC_k$ must tend to infinity, we have that $diam(\CC_k) \raw 0$ and hence also $diam(\psi_+(\CC_k)) \raw 0$. Again $\CC_k \cup \ga_k$ is a connected set whose diameter tends to zero as $k \raw \infty$ and the same is true for $\psi_+(\CC_k \cup \ga_k)$. Since $w_k$ belongs to the closure of $\CC_k \cup \ga_k$ we must have $\psi_+(w_k) \raw \psi_+(w)$, as $k \raw \infty$. 
\end{proof}

In order to extend $\psi_c$ to the full Julia set we need first a lemma.
Recall that $E = \cup_{n=0}^{\infty} f_c^{-n}(0)$. The points in $E$ are bi-accessible and the set $E$ is dense in the Julia set. Denote the full set of bi-accessible points in the Julia set by $\BB iac$, (so $E \subset \BB iac$).
Let
$$\Th = \{ (\th, \th') : \text{  There exist $z \in E$, such that  $\RR_c(\th)$ and $\RR_c(\th')$ land on $z$} \}$$
be the set of angle pairs landing at points in $E$.  We call the pair $(\th,\th')$ a {\em ray pair} for the point $z$. More generally, we say that the angle $\al$ is a landing angle for
$z$ if the ray $\RR_c(\al)$ lands at $z$. If $z \in J(f_c)$ is a multiply accessible point in $J(f_c)$ with landing angles $(\al_1,\ldots,\al_p)$, then we say that
the angles in $(\al_j,\al_{j+1})$ are {\em adjacent} if there is no angle $\al_k$ between them.


\begin{Lem} \label{biac}
Let $w \in J(f_c)$ be a multiply accessible point with landing angles $(\al_1,\ldots,\al_p)$. Put $\al_{p+1}=\al_1$.
Then for each adjacent angle pair $(\al_j,\al_{j+1})$, $j=1, \ldots , p$  there is a sequence $w_k \raw w$, $w_k \in E$ such that the ray pairs $(\th_k,\th_k')$ converge to
$(\al_j,\al_{j+1})$, i.e. $\th_k \raw \al_j$ and $\th_k' \raw \al_{j+1}$, as $k \raw \infty$.
\end{Lem}

\begin{proof}
Suppose the statement of the lemma is not true. That the leaves
$(\al_j,\al_{j+1})$ form a geodesic polygon in the unit disk. First we construct a set of angles $A$ by deleting all arcs $(\th,\th') \in \Th$ apart from the arcs $(\al_j,\al_{j+1})$ (of course) from the unit circle $S^1$. Since each arc is open and its complement non-empty the
intersection of the complement all such (countable) arcs is non empty closed set in $S^1$. We will study the resulting set $A$ of angles and their corresponding external rays which land on $J(f_c)$.

Since the ray-pairs $(\th,\th')$ are partially ordered by inclusion, and since there is an upper bound for each partial ordering (one of the arcs $(\al_j,\al_{j+1})$) there has to be a maximal element in each ordering (by Zorn's Lemma); i.e. there has to be a {\em maximal} ray pair $(\th,\th') \in A$. Note however, that the ray pair $(\th,\th')$ may not belong to $\Th$ anymore. However the rays $\RR_c(\th)$ and $\RR_c(\th')$ must land on the same point and that this point is (at least) bi-accessible. Since the geodesic polygon formed by the angles $(\al_1,\ldots,\al_p)$
partition the unit circle into at least $2$ intervals, and since $\Th$ is dense, the set of maximal elements has to be at least $2$.
Now consider the corresponding points $$B = \{ z \in J(f_c) : \text{ There exist an angle $\th \in A$ such that the ray $\RR_c(\th)$ lands at $z \in J(f_c)$} \}.$$
Clearly, $B$ is a closed connected subset of the Julia set. Note that if $B$ is trivial, i.e. consists of only one single point $w \in J(f_c)$, then the lemma follows, since this would mean that the maximal ray pairs are precisely the arcs $(\al_j, \al_{j+1})$. Hence we suppose to reach a contradiction, that $B$ is not just a single point.

We want to prove that $B$ in fact contains an iterated pre-image of the critical point, i.e. there is a $z \in E \cap B$. Take two points $a,b \in B$ which have the property that they are (at least) bi-accessible. We may choose them such that there are  maximal ray pairs $(\th_a,\th_a')$ and $(\th_b,\th_b')$ with angles in $A$, such that $\RR_c(\th_a)$ and $\RR_c(\th_a')$ land at $a$ and $\RR_c(\th_b), R_c(\th_b')$ land at $b$. If any of them, say $a$, is more than bi-accessible, i.e. there is a larger arc $(\th_1,\th_1')$ whose rays land at $a$ and such that $(\th_a,\th_a') \subset (\th_1,\th_1')$, then consider the largest such arc and call it $(\th_1, \th_1')$ for $a$ and correspondingly $(\th_2,\th_2')$ for $b$.  Now take an
equipotential $Q$ together with the rays $\RR_c(\th_1), \RR_c(\th_1'), \RR_c(\th_2), \RR_c(\th_2')$. These rays together with $Q$ cut out three (bounded) puzzle pieces; choose the one which meets both $a$ and $b$ and call it $P$. By non-normality, there is an $N$ such that $f_c^N(P)$ covers the whole Julia set and hence we must also have $f_c^N(J(f_c) \cap P) = J(f_c)$.

Since $B$ is connected and $J(f_c)$ is uniquely path-connected there is a path $\ga \subset B$ which connects $a$ and $b$. If $f_c^j(\ga)$ does
not meet the critical point for any $j \leq N$ then the map $f_c^N$ is univalent in a neighbourhood of $\ga$, or $f_c^N(\ga)$ creates a loop in
$J(f_c)$ which is impossible (since $J(f_C)$ is uniquely path connected).
The set $J(f_c) \sm f_c^N(\ga)$ is a disjoint union disjoint connected components (since the endpoints $f_c^N(a)$ and $f_c^N(b)$ are still at least bi-accessible).
Let $X$ be one of them meeting $f_c^N(\ga)$ at $f_c^N(a)$ or $f_c^N(b)$, and take some $y \in X$, $y \neq f_c^N(a), y \neq f_c^N(b)$. Let $x \in P \sm \ga$ be a pre-image inside $P$. Then there is a path $\ga' \subset J(f_c)$ connecting $x$ with
the closest point $x' \in \ga$.  Since $f_c^N$ is univalent on $\ga$ there is only one pre-image to $f_c^N(x')$ in $\ga$, namely $x'$ (of course). We have to have $x' \neq a,b$. To see this, suppose that $x'=a$ for instance. Then we claim that $\ga'$ would not lay inside $P$. The rays $\RR_c(\th_1), \RR_c(\th_1')$, by the definition of the angles $\th_1, \th_1'$, cut of all but one of the components of $J(f_c) \sm \{ a \}$. If $x'=a$ then $\ga'$ would not belong to $P$. Hence $x' \neq a$. The argument is similar for $b$ of course. So we have $x' \neq a,b$.

Moreover, $f_c^N(x')$ cannot be $f_c^N(a)$ or $f_c^N(b)$ since these points are images of $a$ and $b$ resp. But this means that $f_c^N(x')$ is connected to $f_c^N(x)$ with a path intersecting $f_c^N(\ga)$, where this intersection itself is a non-trivial path (i.e. consists of more than one point). Any point in $f_c^N(\ga)$ inside a small neighbourhood of $f_c^N(x')$ has to have two pre-images; one pre-image in
$\ga'$ which is disjoint from $\ga$ and also one pre-image in $\ga$. Both pre-images can be chosen arbitrarily close to $x'$ by continuity of $f_c^N$. But this contradicts the fact that $f_c^N$ is univalent in a neighbourhood of $\ga$. Hence there has to be a critical point on $f_c^j(\ga)$ for some $j \leq N$ and hence there has to be a point in $\ga$ (not equal to $a$ or $b$) that belong to $E$. The lemma is thereby proved.

\end{proof}

We now extend $\psi_c$ as follows. Let $z_k\raw z$, as $k \raw \infty$, where $z \in J(f_c)$ and $z_k \in E$. By the previous section $\psi_c(z_k)$ is well defined. We want to put
\[
\psi_c(z) = \lim\limits_{k \raw \infty} \psi_c(z_k).
\]
However, we have to show that this is independent of the choice of the sequence $z_k \in E$. First we note that if $z$ is uni-accessible then there is one single angle $\th$ for which the ray $\RR_c(\th)$ lands at $z$. If $z_k \raw z$ then all the angles of the rays landing at $z_k$ have to converge to $\th$. Suppose that the rays $\RR_+(-\th_k)$ and $\RR_+(-\th_k')$ land at $z_k$. Since  $\th_k, \th_k' \raw \th$ as $k \raw \infty$ we get immediately that $\psi_c(z_k) \raw \psi_c(z)$ by the continuity of $\psi_+$.

Let us assume that $z$ is multiply accessible. Let $\al_1, \ldots, \al_p$ be the angles of the rays landing at $z$. By lemma \ref{biac} there exist a sequence $z_{k,j} \in E$ such that the angles $(\th_{k,j},\th_{k,j}') \raw (\al_j, \al_{j+1})$ for $j = 1,\ldots, p-1$.
Let $w_{k,j}$ be the landing point of the ray $\RR_+(-\th_{k,j})$ and $w_{k,j}'$ the landing point of the ray $\RR_+(-\th_{k,j}')$. By the construction of $\psi_+$ we know that $\psi_+(w_{k,j}) = \psi_+(w_{k,j}')$.
Define
\begin{equation} \label{psic}
\psi_c(z) = \lim\limits_{k \raw \infty} \psi_+(w_{k,j}) = \lim\limits_{k \raw \infty} \psi_+(w_{k,j}').
\end{equation}
Now, repeat this argument with $j$ replaced by $j+1$, starting with $j=1$. This makes $\psi_c$ well defined, and independent of the choice of $j \in \{1,\ldots,p\}$ above. Since $\psi_c$ is constructed via $\psi_+$, the function $\psi_c$ is continuous.

\section{Ray equivalence and proof of Theorem \ref{main}}

The semi-conjugacies $\psi_+$ and $\psi_c$ are now constructed. What remains is to prove that they satisfy the ray equivalence relation.

Put $\psi_+=\psi_1$ and $\psi_c =\psi_2$. In this section we show that the maps $\psi_+$ and $\psi_c$ satisfy the ray equivalence relation:
\begin{equation} \label{equiv}
\psi_j(w) = \psi_k(z) \text{  if and only if   }  z \sim_r w,
\end{equation}
$j,k = 1,2$.

By construction, (\ref{equiv}) is satisfied if $z \in E$, $k=2$ and $w \in \BB iac_+$, $j=1$, where $\BB iac_+$ is the set of bi-accessible points in the basilica. Hence in the cases below we always suppose that any $z \in J(f_c)$ does not belong to $E$ and any $w \in J(f_{-1})$ is not bi-accessible.

Moreover, the construction of $\psi_+$ and $\psi_c$ almost immediately gives the first implication; $z \sim_r w$ implies $\psi_j(w) = \psi_k(z)$ for any $j,k=1,2$. In case 1 below we outline the arguments for $j=k=1$. For the other cases, when not $j=k=1$, we leave the details to the reader. The other implications also follow from the construction, but as obvious. 

{\bf Case 1. $j=k=1$} We first suppose that $w_1 \sim_r w_2$ (and both points lay in the basilica) and there are rays $\RR_+(t_1)$ which land on $w_1$ and $\RR_+(t_2)$ which lands at $w_2$ such that $\RR_c(-t_1) $ and $\RR_c(-t_2)$ land at the same point $z \in J(f_c)$. The fact that $\psi_+(w_1) = \psi_+(w_2)$ follows from (\ref{psic}). If there is a larger chain of rays connecting $w_1$ with $w_2$ the same argument applies finitely many times.

Now suppose $\psi_+(w_1) = \psi_+(w_2)$ and suppose that $w_1$ and $w_2$ are not ray equivalent to reach a contradiction. Let $\RR_+(t_1)$ and $\RR_+(t_2)$ be two rays landing at $w_1$ and $w_2$ respectively. Consider the rays $\RR_c(-t_1)$ and $\RR_c(-t_2)$ and suppose they land at two distinct points $z_1 \in J_c$ and $z_2 \in J_c$ respectively. Let $A=\{ \al_1, \ldots, \al_p\}$ and $B=\{\be_1, \ldots, \be_q\}$ be the angles for the rays landing at $z_1$ and $z_2$ respectively, and so $-t_1 \in A$ and $-t_2 \in B$. It is clear that some adjacent ray pair $(\al,\al' )$ in $A$ has to contain all angles in $B$ unless $z_1$ is uni-accessible, in which case we put $\al'=\al+1$ (regarding $S^1$ as $\R / \Z$). We may assume that at least one of $w_1$ or $w_2$ is a tip in the basilica, since otherwise $\psi_+(w_1)=\psi_+(w_2)$ only at certain pre-images of the $\al$-fixed point and these points are ray equivalent by construction. 
Let $z$ be a point in the basilica and let $\BB(t)$ be the (possibly finite) bubble ray in the basilica which, in the finite case, ends at a bubble containing $z$ on its boundary where $\RR_+(t)$ lands on $z$ or, in the infinite case, lands at the landing point of $\RR_+(t)$. In the finite case we also say that $\BB(t)$ lands at $z$. In the infinite case clearly $z$ is a tip. Consider the bubble rays $\BB_R(-\al') = \psi_+(\BB(-\al)), \BB_R(-\al') = \psi_+(\BB(-\al'))$ and $\BB_R(-\be) = \psi_+(\BB(-\be))$, where we may suppose that the latter is infinite for some $\be \in B$.

By Lemma \ref{biac} there are ray-pairs $(\th_k, \th_k')$ that converges to $(\al,\al')$. 
Let us suppose that $\al < \th_k < \th_k' < \al'$ and that $\al < \be < \al'$. For some $k_0$ we have to have $\th_k < \be < \th_k'$ for all $k \geq k_0$. Suppose that $\RR_+(-\th_k)$ and $\RR_+(-\th_k')$ land at $w_k$ and $w_k'$ respectively. Recall that $\psi_+(w_k) = \psi_+(w_k')$ by construction. If $\psi_+(w_1)=\psi_+(w_2)$ this means that $\BB_R(-\be)$ also lands at $\psi_+(w_1)$. 
Since $\al < \th_k < \be < \th_k' <\al'$, it follows that the finite (closed) bubble rays $\BB_R(-\th_k)$ and $\BB_R(-\th_k')$ enclose a closed region where $\BB_R(-\be)$ is trapped inside. So the only possibility that $\BB_R(-\be)$ lands at $\psi_+(w_1)$ is that $\psi_+(w_1) = \psi_+(w_k)=\psi_+(w_k')$, which is impossible by Lemma \ref{tips}, since $w_k$ and $w_k'$ are pre-images of $p_{\al}$. Hence $z_1=z_2$ and we are done.

{\bf Case 2. $j=1, k=2$}
Again, the implication $z \sim_r w$ implies $\psi_1(w) = \psi_2(z)$ follows automatically from the construction.

Let us therefore assume $\psi_c(w) = \psi_+(z)$, and not $z \sim_r w$. But then there are rays $\RR_c(t)$ landing at $z$ and $\RR_+(s)$ landing at $w$ such that $\psi_+(-t)$ lands at some $w' \neq w$ and $\RR_c(-s)$ lands at some $z' \neq z$ (otherwise $z \sim_r w$). But then $\psi_+(w) = \psi_+(w')$ and we can apply Case 1.

{\bf Case 3. $j=k=2$}.
Again, if $z_1 \sim_r z_2$ then $\psi_c(z_1) = \psi_c(z_2)$ follows from the construction.

Suppose now that $\psi_c(z_1) = \psi_c(z_2)$. We have to prove that $z_1 \sim_r z_2$. Suppose $z_1$ and $z_2$ are not ray equivalent. Then there are rays $\RR_c(t_j)$ landing at $z_j$, $j=1,2$ such that the rays $\RR_+(-t_j)$ land at two different points $w_1$ and $w_2$. But then $\psi_+(w_1) = \psi_+(w_2)$ so we can apply Case 1 again. This finishes the proof of Theorem \ref{main}.

\newpage

\bibliographystyle{plain}
\bibliography{ref}

\begin{thebibliography}{1}

\bibitem{Asp-Roesch}
{M}agnus {A}spenberg and {P}ascale {R}oesch.
\newblock {N}ewton maps as matings of cubic polynomials.
\newblock {\em Proc. London Math. Soc.}, 113:77--112, 2016 (1).

\bibitem{AY}
Magnus Aspenberg and Michael Yampolsky.
\newblock {M}ating non-renormalisable quadratic polynomials.
\newblock {\em Comm. Math. Phys.}, 287(1):1--40, 2009.

\bibitem{Dudko}
{D}imitry {D}udko.
\newblock {M}atings with laminations.
\newblock Preprint, arXiv: 1112.4780.

\bibitem{Lars}
{L}ars {P}edersen.
\newblock {A}n analysis of a shared mating in {$V_2$}.
\newblock Master Thesis, Ume\r{a} University, 2014.

\bibitem{MS}
Mitsuhiro Shishikura.
\newblock On a theorem of {M}. {R}ees for matings of polynomials.
\newblock In {\em The Mandelbrot set, theme and variations}, volume 274 of {\em
  London Math. Soc. Lecture Note Ser.}, pages 289--305. Cambridge Univ. Press,
  Cambridge, 2000.

\bibitem{TL}
Lei Tan.
\newblock Matings of quadratic polynomials.
\newblock {\em Ergodic Theory Dynam. Systems}, 12(3):589--620, 1992.

\bibitem{Timorin}
Vladlen Timorin.
\newblock The external boundary of {M}.
\newblock In {\em Holomorphic Dynamics and Renormalization: A Volume in Honour
  of John Milnor's 75th Birthday}, volume~53 of {\em Fields Institute
  Communicatios}, pages 225--268. AMS and the Fields Institute, 2008.

\end{thebibliography}

\end{document}